\documentclass[reqno]{amsart}
\pdfoutput=1
\usepackage{amssymb}
\usepackage{amsthm}
\usepackage{bbm}
\usepackage{graphicx}

\usepackage[colorlinks=true]{hyperref}
\usepackage[all]{xypic}

\usepackage[final]{epsfig}
\usepackage{psfrag}
\usepackage{epstopdf}

\newtheorem{thm}{Theorem}[section]

\newtheorem{proposition}[thm]{Proposition}

\newtheorem{lemma}[thm]{Lemma}

\newtheorem{lem}[thm]{Lemma}
\newtheorem{prop}[thm]{Proposition}

\theoremstyle{definition}
\newtheorem{definition}{Definition}[section]

\theoremstyle{observation}

\theoremstyle{definition}

\newcommand{\imp}{\frak{I}}
\newcommand{\pimp}{\frak{P}}

\newcommand{\A}{\mathcal{A}}

\newcommand{\bd}{\partial}

\newcommand{\be}{\beta}

\newcommand{\uc}{\mathbb{S}^1}

\newcommand{\defin}[1]{{\it #1}}

\newcommand{\Q}{\mathbb{Q}}

\theoremstyle{theorem}

{\bf}{\it}
{\bf}{\it}
\newtheorem*{riemann-mapping-theorem}{Riemann Mapping  Theorem}
{\bf}{\it}
{\bf}{\it}
{\bf}{\it}
{\bf}{\it}
{\bf}{\it}
{\bf}{\it}
{\bf}{\it}
\newtheorem*{carat}{Carath{\'e}odory Theorem}{\bf}{\it}
\newtheorem*{baby-carat}{Carath{\'e}odory Theorem for locally connected domains}{\bf}{\it}
{\bf}{\it}
{\bf}{\it}

\theoremstyle{remark}

\newenvironment{pf}{\proof[\proofname]}{\endproof}
\newenvironment{pf*}[1]{\proof[#1]}{\endproof}
\usepackage{euscript}

\newcommand{\cal}[1]{{\mathcal #1}}

\newcommand{\beq}{\begin{equation}}
\newcommand{\eeq}{\end{equation}}

\newtheorem{defn}{Definition}[section]

\newcommand{\diam}{\operatorname{diam}}
\newcommand{\length}{\operatorname{length}}
\newcommand{\dist}{\operatorname{dist}}

\renewcommand{\mod}{\operatorname{mod}}
\newcommand{\tl}{\tilde}

\newcommand{\eps}{\epsilon}

\numberwithin{equation}{section}
\newcommand{\thmref}[1]{Theorem~\ref{#1}}
\newcommand{\propref}[1]{Proposition~\ref{#1}}

\newcommand{\lemref}[1]{Lemma~\ref{#1}}

\newcommand{\cA}{{\cal A}}

\newcommand{\cM}{{\cal M}}

\newcommand{\cB}{{\cal B}}

\newcommand{\cO}{{\cal O}}

\newcommand{\CC}{{\mathbb C}}
\newcommand{\RR}{{\mathbb R}}

\newcommand{\TT}{{\mathbb T}}
\newcommand{\ZZ}{{\mathbb Z}}
\newcommand{\NN}{{\mathbb N}}
\newcommand{\DD}{{\mathbb D}}

\newcommand{\QQ}{{\mathbb Q}}

\newcommand{\ignore}[1]{{}}

\newcommand{\inter}[1]{\overset{\circ}{#1}}

\title[Non-computable impressions]{Non-Computable impressions of computable external rays of quadratic polynomials.}
\author{Ilia Binder, Cristobal Rojas, Michael Yampolsky}
\thanks{I.B. and M.Y. were partially supported by NSERC Discovery Grants. C.R. was partially supported by Proyecto FONDECYT  No. 11110226. }
\date{February 3, 2014}
\begin{document}
\begin{abstract}
We discuss computability of impressions of prime ends of compact sets. In particular, we 
construct  quadratic Julia sets which possess explicitly described non-computable impressions.
\end{abstract}

\maketitle
\section{Introduction}

Informally speaking, a compact subset of the plane is computable if there exists an algorithm to visualize it on a computer screen with an arbitrary given resolution. Of central interest in applications to Complex Dynamics is the question of computability of the
Julia set of a rational mapping. It is known (\cite{BY-book,BY-MMJ}) that there exist quadratic polynomials $f_c(z)=z^2+c$
with computable coefficients $c$ and  with non-computable Julia sets $J_c$. This non-computability phenomenon is quite subtle.
In particular,  the filled Julia set $K_c$
is computable \cite{BY-book}, and,
moreover, the harmonic measure $\omega_c$ of the Julia set is computable \cite{BBRY}. 
Thus the parts of the Julia set which are hard to compute are ``inward pointing'' decorations, forming narrow fjords of $K_c$.
If the fjords are narrow enough, they will not appear in a finite-resolution image of $K_c$, which explains how the former can
be computable even when $J_c$ is not. Furthermore, a very small portion of the harmonic measure resides in the fjords, again
explaining why it is always possible to compute the harmonic measure.

Suppose the Julia set $J_c$ is connected, and denote 
$$\phi_c:\hat\CC\setminus \overline{\DD}\to\hat\CC\setminus K_c$$
 the unique conformal mapping satisfying the normalization $\phi_c(\infty)=\infty$ and $\phi_c'(\infty)=1$.
Carath{\'e}odory Theory (see e.g. \cite{Mil} for an exposition) implies that $\phi_c$ extends continuously to map 
the unit circle $\uc$ onto the {\it Carath{\'e}odory completion} $\hat J_c$ of the Julia set. 
An element of the set $\hat J_c$ is a {\it prime end} $p$ of $\CC\setminus K_c$. The impression $\imp(p)$ of a prime end 
is a subset of $J_c$ which should roughly be thought as a part of $K_c$ accessible by a particular approach from the exterior.
The harmonic measure $\omega_c$ can be viewed as the pushforward of the Lebesgue measure on $\uc$ onto the set of prime end
impressions.

In view of the above quoted results, from the point of view of computability, prime end impressions should be seen as borderline
objects. On the one hand, they are subsets of the Julia set, which may be non-computable, on the other they are ``visible from
infinity'', and as we have seen accessibility from infinity generally implies computability.

It is thus natural to ask:

\medskip
\noindent
{\bf Question 1.} {\sl Is the impression of a prime end of $\hat\CC\setminus K_c$ always computable?}

\medskip
\noindent
To formalize the above question, we need to describe a way of specifying a prime end. We recall that {\it the external ray}
$R_\alpha$ of angle $\alpha\in \RR/\ZZ$ is the image under $\phi_c$ of the radial line $\{re^{2\pi i\alpha}\;:\;r>1\}$.
The curve $$R_\alpha=\phi_c(\{re^{2\pi i\alpha}\;:\;r>1\})$$
lies in $\hat\CC\setminus K_c$. The {\it principal impression of an external ray} $\pimp(R_\alpha)$ is the set of limit points of $\phi_c(re^{2\pi i\alpha})$ as $r\to 1$.
If the principal impression of $R_\alpha$ is a single point $z$, we say that $R_\alpha$ {\it lands} at $z$. External rays play a very important role in the study of polynomial dynamics.
 
It is evident that every principal impression is contained in the impression of a unique prime end. We call the impression of this prime end the 
{\it prime end impression of an external ray} and denote it
$\imp(R_\alpha)$. A natural refinement of Question 1 is the following:

\medskip
\noindent
{\bf Question 2.} {\sl Suppose $\alpha$ is a computable angle. 
Is the prime end impression $\imp(R_\alpha)$ computable?}

\medskip
\noindent
The purpose of this paper is to prove that the answer is emphatically negative: 

\bigskip

\noindent\textbf{Main Theorem.} There exists a computable complex parameter $c$ and a computable Cantor set of angles 
$C\subset \uc$ such that for every angle $\alpha \in C$, the impression $\imp(R_{\alpha})\subset J_c$ is not computable.
Moreover, any compact subset $K\Subset J_c$ which contains $\imp(R_\alpha)$ is non-computable.
\bigskip

\noindent
{\bf Acknowledgment.} The authors would like to thank Sasha Blokh for several stimulating discussions of this and related topics.


 \section{A brief introduction to Computability}
We give a very brief summary of relevant notions of Computability Theory and Computable Analysis. For a more in-depth
introduction, the reader is referred to \cite{BY-book,BBRY}.
As is standard in Computer Science, we formalize the notion of
an algorithm as a {\it Turing Machine} \cite{Tur}.
It is more intuitively familiar, and provably equivalent,
to think of an algorithm as a program written in any standard programming language.
In any programming language there is only a countable number of possible algorithms. Fixing the language, we can
enumerate them all (for instance, lexicographically). Given such an ordered list $(\cA_n)_{n=1}^\infty$ of all
algorithms, the index $n$ is usually called the \emph{G\"odel number} of the algorithm $\cA_{n}$.

We will call a  function $f:\NN\to\NN$  \emph{computable} (or {\it recursive}), if there exists an algorithm $\cA$ which, upon input $n$, outputs $f(n)$.
A set $E\subseteq \NN$ is said to be \defin{computable} (or {\it recursive}) if its characteristic function
$\chi_E:\NN\to\{0,1\}$ is computable.

Since there are only countably many algorithms, there exist only countably many computable subsets of $\NN$. A well known ``explicit'' example of a non computable set is given by the \emph{Halting set}
$$H:=\{i \text{ such that }\cA_{i}\text{ halts}\}.$$
 Turing \cite{Tur} has shown that  there is no algorithmic procedure to decide, for any $i\in\NN$, whether or not the algorithm with G\"odel number $i$, $\cA_{i}$, will eventually halt.

Extending algorithmic notions to functions of real numbers was pioneered by Banach and Mazur \cite{BM,Maz}, and
is now known under the name of {\it Computable Analysis}. Let us begin by giving the modern definition of the notion of computable real
number,  which goes back to the seminal paper of Turing \cite{Tur}.

\begin{definition}A real number $x$ is called

\begin{itemize}
\item \defin{computable} if there is a computable function $f:\NN \to \QQ$ such that $$|f(n)-x|<2^{-n};$$
\item \defin{lower-computable} if there is a computable function $f:\NN \to \QQ$ such that
$$f(n)\nearrow x;$$
\item \defin{upper-computable} if there is a computable function $f:\NN \to \QQ$ such that
$$f(n)\searrow x.$$
\end{itemize}
\end{definition}

	Algebraic numbers or  the familiar constants such as $\pi$, $e$, or the Feigembaum (\cite{Hoy09}) constant are all computable. However, the set of all computable numbers $\RR_C$ is necessarily
countable, as there are only countably many computable functions. Lower  (or upper)-computable numbers are also called \emph{left (or right)}-computable. 
It is straightforward to see that a number is computable if it is simultaneously left- and right-computable.
It is easy to present an example of a non-computable left- or right-computable real number. For instance, define the Halting predicate
$p(i)$ to be equal to $1$ if $A_i$ halts and $0$ otherwise. The number
$$\alpha=\sum_{n=1}^\infty 10^{-n}p(n)$$
is evidently not computable. To see that it is left computable, let $p(i,j)$ be the predicate expressing the truth of the sentence
``$A_i$ halts on step $j$'', and set
$$\alpha_k=\sum_{n=1}^k\sum_{j=1}^kp(n,j).$$
Naturally, $\alpha_k\nearrow \alpha.$

For more general objects, computability is typically defined according to the following principle: object $x$ is computable if there exists an algorithm $\A$ which, upon input $n$, outputs a finite suitable description of $x$  at precision  $n$. In this case we say that algorithm $\A$ \emph{computes} object $x$. 

For instance, computability of compact subsets of $\RR^n$ is defined as follows. Recall that {\it Hausdorff distance} between two
compact sets $K_1$, $K_2$ is
$$\dist_H(K_1,K_2)=\inf_\eps\{K_1\subset U_\eps(K_2)\text{ and }K_2\subset U_\eps(K_1)\},$$
where $U_\eps(K)=\bigcup_{z\in K}B(z,\eps)$ stands for an $\eps$-neighbourhood of a set.

We say that $K\Subset\RR^n$ is {\it computable} if there exists an algorithm $A$ with a single input $n\in\NN$ which outputs a
finite set $C_n$
of points with rational coordinates such that
$$\dist_H(C_n,K)<2^{-n}.$$

An equivalent, and more intuitive, way of defining a computable set is the following. Let us say that a {\it pixel} is a dyadic
cube with side $2^{-n}$ and dyadic rational vertices. A set $K$ is computable if there exists an algorithm $A$ which given a pixel
with side $2^{-n}$ outputs $0$ if the center of the pixel is at least $2\cdot 2^{-n}$-far from $K$, outputs $1$ is the center is at most
$2^{-n}$-far from $K$, and outputs either $0$ or $1$ in the ``borderline'' case.

In this paper
we will speak of {\it uniform computability} whenever a group of computable objects (functions, sets, etc) is computed by a single
algorithm:

\medskip
\noindent
\emph{the objects $\{x_\gamma\}_{\gamma\in\Gamma}$ are computable \defin{uniformly on
a countable set $\Gamma$} if there exists an algorithm $\cA$ with an input $\gamma\in\Gamma$, such that for all
$\gamma\in \Gamma$, $\cA_\gamma:=\cA(\gamma,\cdot)$ computes $x_{\gamma}$}.\\

For instance, a sequence $x_{n}$ of computable points is \emph{uniformly computable} if there is a single algorithm $\A$ which for every $n$ and $m$ outputs a rational number satisfying $|\A(n,m)-x_{n}|<2^{-m}$. 

To define a computable real-valued function we need to introduce another notion. We say that a function $\phi:\NN\to\QQ$ is an {\it oracle} for $x\in \RR$ if for every $m\in \NN$
$$d(\phi(m),x)<2^{-m}$$

On each step, an algorithm may {\it query} an oracle by reading the value of the function $\phi$ for an arbitrary $m\in\NN$. 

Let $S\subset \RR$. Then a function $f:S\to \RR$ is called {\it computable} if
there exists an algorithm $\cA$ with an oracle for $x\in S$ and an input $n\in\NN$ which outputs a rational number $s_n$
such that $|s_n-f(x)|<2^{-n}.$
In other words, given an arbitrarily good approximation of the input of $f$ it is possible to constructively
approximate the value of $f$  with any desired precision. Open sets can be described by means of {\it rational balls}: balls with rational centres and radii. An open set $A\subset \RR$ is called {\it lower-computable} if it is the union of a computable sequence of rational balls.  It is easy to see that a function $f$ is computable if the preimages of rational balls are uniformly lower-computable open sets.  Computability of functions and open sets of $\CC$, $\RR^n$, etc, is defined in
a similar fashion.

The following well known characterization of computable compact sets will be used in the sequel. 

\begin{proposition}\label{computableset}A compact set $K$ is computable if and only if $(i)$ there is a sequence $x_{n}\in K$ of uniformly computable points which is dense in $K$ and $(ii)$ the complement $K^{c}$ is a lower-computable open set.  
\end{proposition}


\medskip
\noindent

\section{An example of a computable set with a non-computable  impression.}

We refer the reader to \cite{Mil} for a detailed exposition of Carath{\'e}odory Theory of prime ends. Here we briefly recall
the main definitions. Let $\Omega\subset\hat \CC$ be a connected domain. Arbitrarily fix a base point $\omega\in\Omega$.
A {\it crosscut} $\gamma$ is the image of a simple curve
$$\tl\gamma:[0,1]\to \hat\CC$$
with the properties
$$\tl\gamma:(0,1)\to\Omega,\; \tl\gamma(0)\neq\tl\gamma(1),\text{ and }\tl\gamma(\{0,1\})\subset\partial\Omega.$$
We call the image of the open interval $(0,1)$ {\it the interior} of $\gamma$ and denote it $\inter{\gamma}$.

For a crosscut $\gamma$ the {\it crosscut neighborhood} $N_\gamma$ will denote the connected component of $\Omega\setminus \gamma$
which does not contain the base point $\omega$.

A {\it fundamental chain} is a sequence of crosscuts $\{\gamma_i\}_{i=1}^\infty$ which satisfies the following properties:
\begin{itemize}
\item for all $i$, $\inter{\gamma}_i\subset N_{\gamma_{i-1}}$;
\item $\diam \gamma_i\to 0$.
\end{itemize}
Two fundamental chains $\{\gamma_i\}$ and $\{\kappa_j\}$ are equivalent if for every $i$ there exists $j$ such that
$$N_{\kappa_j}\supset {\inter{\gamma}}_i,$$
and vice versa.
An equivalence class of fundamental chains is called a {\it prime end}. The {\it impression} $\imp(\mathbf p)$ of a prime end $\mathbf p$ is the intersection
$$\imp(\mathbf p)=\cap \overline{N_{\gamma_i}}$$
for any fundamental chain representative $\{\gamma_i\}$ of the equivalence class $\mathbf p$.
The space of prime ends possesses a natural topology, an open set in which is specified by a crosscut neighborhood $\gamma$. It forms
the {\it Carath{\'e}odory boundary} $\hat\partial\Omega$ of the domain $\Omega$; together, $\hat\partial\Omega$ and $\Omega$ form the
{\it Carath{\'e}odory closure} $\hat\Omega$.

If the domain $\Omega$ is simply-connected, and its complement contains at least two points, then $\hat\Omega$ is homeomorphic to the
closed unit disk $\overline{\DD}$. In this case, denote
$$\varphi:\Omega\to\DD$$
the conformal Riemann mapping, normalized so that $\varphi(\omega)=0$ and $\varphi'(\omega)>0.$ The key statement of Carath{\'e}odory Theory is:

\begin{carat}
The map $\varphi$ extends to a homeomorphism between $\hat\Omega$ and $\overline{\DD}.$
\end{carat}

We will need the following quantitative version of Carath{\'e}odory Theorem, due to
 Lavrientiev (see \cite{BRY}, Proposition 6.1):
\begin{thm}\label{lavr} Let $\Omega\subset \hat\CC$ be a simply-connected domain, whose complement contains at least two points,
 with $\infty\not\in \Omega$, and let $\omega\in\Omega$ be a
base point. Let $\gamma$ be a crosscut of $ \Omega$, such that $\dist(\gamma,  \omega)\geq M$, for some $M>0$ and $ N_\gamma$ be the component of $ \Omega\setminus\gamma$ not containing $ \omega$. Assume that $\epsilon^2<M/4$. Then
$$\diam(\gamma)\leq \epsilon^{2} \implies \diam(\varphi( N_\gamma))\leq \frac{30\epsilon}{\sqrt{M}}.$$
\end{thm}

Let $K$ be a compact connected subset of $\CC$ with a connected complement, which contains at least two points,
set $$\Omega\equiv\hat\CC\setminus K.$$
Denote $\phi$ the conformal homeomorphism $$\phi:\hat\CC\setminus \overline{\DD}\to \Omega$$
normalized by $\phi(\infty)=\infty$ and $\phi'(\infty)=1$.

As before, for $\alpha\in\RR/\ZZ$ we let the external ray $R_\alpha$
be the image
$$R_\alpha=\phi(\{re^{2\pi i\alpha}\;:\;r>1 \}),$$
and we say that the {\it principal impression} $\pimp(R_\alpha)$ is the set of limit points of $\phi(re^{2\pi i\alpha})$ as $r\searrow 1$.
If $\pimp(R_\alpha)$ is a single point $a\in \partial K$, then we say that the ray $R_\alpha$ {\it lands} at $a$.
In this case, $\alpha$ is an {\it external angle }of $a$. We will refer to the unique prime end impression containing $\pimp(R_\alpha)$ as
the {\it prime end impression of } $R_\alpha$ and denote it $\imp(R_\alpha)$.
Vice versa, given a prime end $\mathbf p$, there exists a unique external ray $R_\alpha$ with $\pimp(R_\alpha)\subset \imp(\mathbf p).$
We will thus speak of the principal impression of $\mathbf p$ and write
$$\pimp(\mathbf p)=\pimp(R_\alpha).$$

We note the following theorem due to Lindel\"of Theorem (see, for example, \cite{PomUnivalent}, Theorem 9.7):
\begin{thm}
\label{lindel}
Let $\mathbf p$ be a prime end.
Then
$z\in \pimp(\mathbf p)$ if and only if there exists a fundamental chain representative  $\{\gamma_i\}$ of $\mathbf p$ such that
$$\dist(z,\gamma_i)\to 0,$$
that is, $z$ is a limit point of a sequence $\gamma_i$.
\end{thm}

Let us pick $a<b<1/2$, such that $a$ is a non-computable lower-computable number and $b$ is a non-computable upper-computable number. Let $a_n\uparrow a$, $1/2>b_n\downarrow b$ be two computable sequences converging to $a$ and $b$ respectively.

Let $Q$ denote the square $[-1,1]\times[-1,1]$. Let $S_n\subset Q$ be a rectangle of height $2\cdot 3^{-n}$ given by
$$ S_n\equiv (-b_n,b_n)\times(3^{-n}, 3^{1-n}],$$
and let $L_n$ and $R_n$ be two sub-rectangles of $S_n$ of height $3^{-(n+1)}$ given by
$$L_n\equiv [-b_n, a_n]\times[8\cdot3^{-n-1}, 3^{1-n}],\; R_n\equiv [-a_n, b_n]\times[5\cdot3^{-n-1}, 2\cdot3^{-n}].$$

Define our domain $\Omega$ as
$$\Omega:=\left(\hat\CC\setminus Q\right)\cup\cup_{n\geq1}\left(S_n\setminus (L_n\cup R_n)\right).$$
See Figure \ref{fig-example} for an illustration of the construction.

\begin{figure}
\centerline{\includegraphics[width=\textwidth]{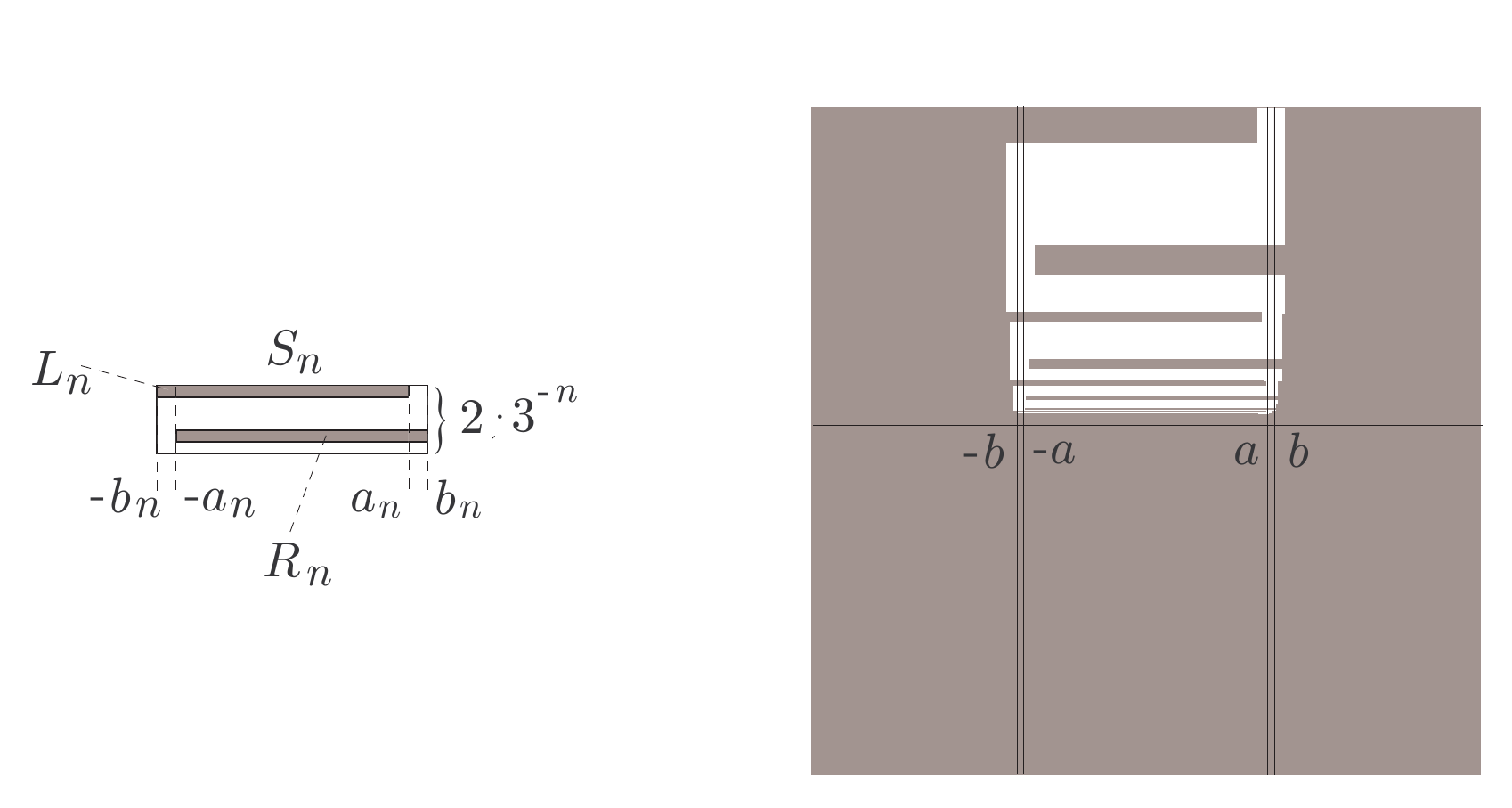}}
\caption{\label{fig-example}Left: a ``building block'' $S_n\setminus (L_n\cup R_n)$; right: the domain $\Omega$.
}
\end{figure}
Fix a basepoint $\omega\in\Omega$ outside of $Q$, and
define the prime end $\mathbf p$ of $\Omega$ by the sequence $\{\gamma_n\}$, where
$$\gamma_n={0}\times[2\cdot3^{-n}, 8\cdot3^{-n-1}]$$ is a fundamental chain of crosscuts.
It is evident that
$$\imp(\mathbf p)=[-b,b]\times\{ 0\}.$$
Furthermore, by Lindel\"of Theorem \ref{lindel},
$$\pimp(\mathbf p)=[-a,a]\times\{0\}.$$
Hence both $\imp(\mathbf p)$ and $\pimp(\mathbf p)$ are not computable.

On the other hand, $\partial \Omega$ is computable, since it
 can be approximated with  precision $3^{-n}$ in Hausdorff metric by a computable polygonal curve
$$\Gamma_n=\partial\left(\left(\hat\CC\setminus Q\right)\cup\cup_{n\geq k\geq1}\left(S_k\setminus (L_k\cup R_k)\right)\right).$$
Note that $\partial\Omega\,\setminus \, [-b,b]\times\{0\}$ can be parameterized by a computable map $f:(0,1)\to\CC$.

Moreover, the external angle $\alpha$, such that
$$\pimp(R_\alpha)=\pimp(\mathbf p)$$
is also computable. To see this, note that $\alpha_n$, the external angle of the point
$$\omega_n=(0, 7\cdot3^{-n-1})\in \Omega$$ is computable by the Constructive Riemann Mapping Theorem (see \cite{BBY3}).
Set $$\Omega_n\equiv N_{\gamma_n}.$$
By Lavrentiev's Theorem
\ref{lavr} we have, after applying a Moebious transformation so as to have $\infty \notin \Omega$, that 
$$\diam\left(\phi^{-1}(\Omega_n)\right)\leq C\sqrt{\length\gamma_n}\leq C 3^{-n/2},$$
for some computable  constant $C$.

Note that $\omega_n\in\gamma_n$. Thus the above estimate gives
$$|\alpha-\alpha_n|\leq C 3^{-n/2},$$
implying the computability of  $\alpha$.

Thus, we have produced an example of a domain with computable boundary and a computable external angle $\alpha$ for which we can not compute either $\imp(R_\alpha)$ or $\pimp(R_\alpha)$.

\section{Siegel disks in the quadratic family and computability of Julia sets}
\label{sec:intro-dyn}
As a general reference on Julia sets of rational maps we refer the reader to the excellent book of J.~Milnor \cite{Mil}.
Here we briefly review the relevant results on computability of quadratic Julia sets, following \cite{BY-book}.

We recall, that the Julia set $J_c$ of a quadratic polynomial $f_c(z)=z^2+c$ is computable if there exists an algorithm $A$,
with an oracle for $c$, which computes this set. That is, $A$ takes a single input $n\in\NN$, may query the value of $c$ with
an arbitrary finite precision, and outputs a finite set of rational points $C_n$ such that
$$\dist_H(J_c,C_n)<2^{-n}.$$
The oracle formulation separates the issue of computing the value of $c$ from the problem of computing $J_c$ when $c$ is known.
In some of the results quoted below, the value of $c$ will itself be a computable complex number, and hence, the computability of
$J_c$ will be equivalent to its computability as a compact subset of $\CC$ (without an oracle).

Let $z_0$ be a periodic point of $f_c$ with period $p$ with multiplier $\lambda\neq 0$.
We say that $f_c$ is {\it locally linearlizable at $z_0$} if there exists a neighborhood $U(z_0)$ and a conformal change of
variable $$\psi:U(z_0)\to B(0,r)$$
with the property
$$\psi\circ f_c^p\circ \psi^{-1}(z)=\lambda z.$$
In the case when $|\lambda|\neq 1,$ a local linearization always exists by a classic result of Schroeder.

In the {\it parabolic case}, when $\lambda$
is a root of unity, the map $f_c$ is not linearizable.

 The remaining possibility is $\lambda=e^{2\pi i\theta}$ with the internal angle $\theta\notin \QQ$. Here two non-vacuous possibilities  exist:
{\it Cremer case}, when $f_c$ is not linearizable, and {\it Siegel case}, when it is.
In the latter case, there exists a maximal linearization  neighborhood $U(z_0)$ which is called a {\it Siegel disk} of $f_c.$

Note that Fatou-Shishikura inequality implies that $f_c$ has no more than one non-repelling orbit. In particular, there could
be no more than one periodic Siegel disk for $f_c$.

We note (see \cite{BY-book}):
\begin{thm}
If the Julia set $J_c$ is not computable, then $f_c$ necessarily has a Siegel disk.
\end{thm}

In view of this, it will be necessary to recall some facts on the occurrence of Siegel disks in the quadratic family. For
simplicity, let us further specialize to the case of a fixed Siegel disk: we will consider the family
$$P_\theta(z)=e^{2\pi i\theta}z+z^2.$$
 The map $P_\theta$ has a neutral fixed point at the origin with multiplier $e^{2\pi i\theta}$.
If it is linearizable, we will denote the Siegel disk by $\Delta_\theta$.

For a number $\theta\in [0,1)$ denote $[r_1,r_2,\ldots,r_n,\ldots]$, $r_i\in\NN\cup\{\infty\}$ its possibly finite
continued fraction expansion:
\begin{equation}
\label{cfrac}
[r_1,r_2,\ldots,r_n,\ldots]\equiv\cfrac{1}{r_1+\cfrac{1}{r_2+\cfrac{1}{\cdots+\cfrac{1}{r_n+\cdots}}}}
\end{equation}
Such an expansion is defined uniquely if and only if $\theta\notin\QQ$. In this case, the {\it rational
convergents } $p_n/q_n=[r_1,\ldots,r_{n}]$ are the closest rational approximants of $\theta$ among the
numbers with denominators not exceeding $q_n$.

Suppose, $\theta\notin\QQ$ and inductively define $\theta_1=\theta$
and $\theta_{n+1}=\{1/\theta_n\}$. In this way,
$$\theta_n=[r_{n},r_{n+1},r_{n+2},\ldots].$$
We define the {\it Yoccoz's Brjuno function} as
$$\Phi(\theta)=\displaystyle\sum_{n=1}^{\infty}\theta_1\theta_2\cdots\theta_{n-1}\log\frac{1}{\theta_n}.$$
One can verify that $$B(\theta)<\infty\Leftrightarrow \Phi(\theta)<\infty.$$

The celebrated result of Brjuno and Yoccoz states:
\begin{thm}
The map $P_\theta$ has a Siegel point at the origin if and only if $\Phi(\theta)<\infty.$
\end{thm}
We call the irrational numbers $\theta\in \uc$ with $\Phi(\theta)<\infty$ {\it Brjuno numbers}.
The sufficiency of the condition $\Phi(\theta)<\infty$ for linearizability of an arbitrary analytic germ with multiplier
$\lambda=e^{2\pi i\theta}$ was proved by Brjuno \cite{Bru} in 1972 (with a different series $B(\theta)=\sum_n\frac{\log(q_{n+1})}{q_n}$
whose convergence is equivalent to that of $\Phi(\theta)$. In 1987 Yoccoz \cite{Yoc} proved that
the condition $\Phi(\theta)<\infty$ is also necessary in the quadratic family.

It is easy to see that, for example, every Diophantine number $\theta$ satisfies the above condition, so there is a full measure
set of Siegel parameters $\theta$ in $\uc$ (the sufficiency of the Diophantine condition was proved by Siegel \cite{siegel} in 1942.

\begin{figure}[h]
\centerline{\includegraphics[width=0.5\textwidth]{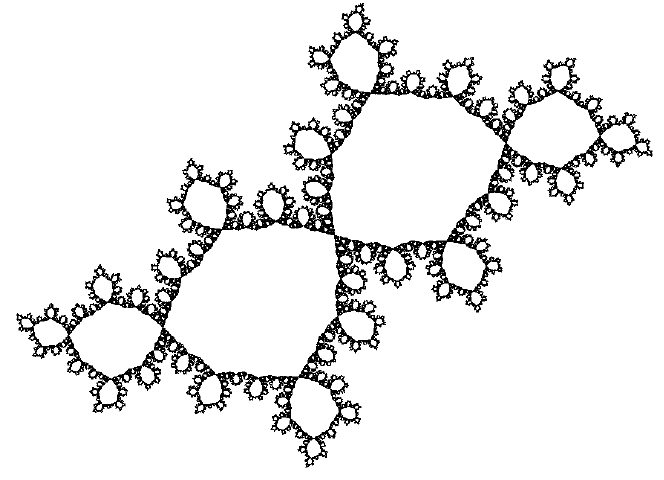}}
\caption{The  Julia set of $P_\theta$ for $\theta=[1,1,1,1,\ldots]$ (the inverse golden mean).}
\label{fig-siegel}
\end{figure}

The proof of Yoccoz's Theorem relies on the following connection between the sum of the series $\Phi$ and the size of the Siegel
disk $\Delta_\theta$.

\begin{defn}
Let $P(\theta)$ be a quadratic polynomial with a Siegel disk $\Delta_\theta\ni 0$. Consider a conformal isomorphism
$\phi:\DD\mapsto\Delta$ fixing $0$. The {\it conformal radius of the Siegel disk $\Delta_\theta$} is
the quantity
$$r(\theta)=|\phi'(0)|.$$
For all other $\theta\in[0,\infty)$ we set $r(\theta)=0$.
\end{defn}

\noindent
By the Koebe One-Quarter Theorem of classical complex analysis, the internal radius of $\Delta_\theta$ is at least
$r(\theta)/4$. Yoccoz \cite{Yoc} has shown that the sum
$$\Phi(\theta)+\log r(\theta)$$
is bounded from below independently of $\theta\in\cB$. Buff and Ch{\'e}ritat have greatly improved this result
by showing that:

\begin{thm}[\cite{BC}]
\label{phi-cont}
The function $\theta\mapsto \Phi(\theta)+\log r(\theta)$ extends to $\RR$ as a 1-periodic continuous
function.
\end{thm}

\noindent
We remark that the following stronger conjecture exists (see \cite{MMY}):

\medskip
\noindent
{\bf Marmi-Moussa-Yoccoz Conjecture.} \cite{MMY} {\it The function $\upsilon: \theta\mapsto \Phi(\theta)+\log r(\theta)$ is H{\"o}lder of exponent $1/2$.}

\medskip
\noindent
From the point of view of computability of $J_\theta=J(P_\theta)$ it is known \cite{BY,BY-book} that
\begin{thm}
The computability of $J_\theta$ with an oracle for $\theta$ is equivalent to the computability of the real number $r(\theta)$,
again, with an oracle for $\theta$.
\end{thm}

The following oracle-less theorem of \cite{BY-MMJ} (see also \cite{BY-book}) implies that there exist {\it computable}
parameters $\theta$ with non-computable $J_\theta$. Let us denote
$$r_*=\displaystyle\sup_{\theta\in \uc} r(\theta).$$
\begin{thm}
\label{nc julia}
Suppose $\theta$ is a computable real number. Then $r(\theta)$ is right-computable.

Conversely, let $r$ be a right computable real number in the interval $(0,r_*)$. Then there exists a computable
parameter $\theta\in \uc$ such that $r(\theta)=r$. Moreover, the value of $\theta$ can be computed uniformly (by an
explicit algorithm) from a computable sequence $r_n\searrow r$.
\end{thm}

Let us make a few notes on topological properties of Siegel Julia sets.
Firstly, the Julia set of any Siegel or Cremer quadratic polynomial is connected.
The following result is due to Sullivan and Douady (see \cite{Sul}):

\begin{thm}
\label{loc con}
If the Julia set of a polynomial mapping $f$ is locally connected, then
$f$ has no Cremer points.
Moreover, every cycle of Siegel disks of $f$ contains at least one
critical point in its boundary.

\end{thm}

\noindent
Thus, in particular, Cremer quadratic Julia sets are {\it never} locally connected. There is a vast
amount of recent work on pathological properties of Cremer quadratics, and we will not attempt to give
a survey of results here.
We cannot offer an illustration
with a Cremer Julia set to the reader -- even though it is known that all such sets are computable,
no informative pictures of them have been produced to this day.

As for Siegel Julia sets the following results are known.
We say that an irrational number
$\theta=[r_0,r_1,r_2,\ldots]$ is of {\it a type bounded by $B$} if $\sup r_i\leq B<\infty$.
The collection of all angles $\theta$ of a bounded type  is the
Diophantine class with exponent $2$.
 Petersen \cite{Pet} showed that $J_\theta$ is locally connected for $\theta$
of  a bounded type.
A different proof of this was later given by the third author \cite{Yam}. Petersen and
Zakeri \cite{PZ} further extended this result to a set of angles $\theta$ which has a full measure in $\TT$.
In \cite{BY-lc} it was shown that there exist computable parameters $\theta$ for which the Julia set $J_\theta$
is locally connected, and yet not computable (see also \cite{BY-book} for an expository account).

On the other hand, Herman in 1986 presented first examples of $P_\theta$ with a Siegel disk whose boundary
does not contain any critical points. By \thmref{loc con} the Julia set of such a map is not locally-connected.
In the papers of Buff-Ch{\'e}ritat \cite{BC2}, and Avila-Buff-Ch{\'e}ritat \cite{ABC} it is shown that the boundary
$\partial \Delta_\theta$ of the Siegel disk itself can be $C^\infty$ smooth. A variation of their argument will be used
in this paper. Note, that if the boundary of $\Delta_\theta$ is a  smooth curve (differentiability at every point is
enough), then it cannot contain the critical point, and hence $J_\theta$ cannot be locally connected.

For future reference let us state several facts on the dependence of the conformal radius of a Siegel disk on the parameter
(details can be found in \cite{BY-book}).

\begin{defn}
Let $(U_n,u_n)$ be a sequence of topological disks $U_n\subset\CC$ with marked points $u_n\in U_n$.
The {\it kernel} or {\it Carath{\'e}odory} convergence $(U_n,u_n)\to (U,u)$ means the following:
\begin{itemize}
\item $u_n\to u$;
\item for any compact $K\subset U$ and for all $n$ sufficiently large, $K\subset U_n$;
\item for any open connected set $W\ni u$, if $W\subset U_n$ for infinitely many $n$, then $W\subset U$.
\end{itemize}
\end{defn}

\noindent
The topology on the set of pointed domains which corresponds to the above definition of convergence is again
called {\it kernel} or {\it Carath{\'e}odory} topology. The meaning of this topology is as follows.
For a pointed domain $(U,u)$ denote
$$\phi_{(U,u)}:\DD\to U$$
the unique conformal isomorphism with $\phi_{(U,u)}(0)=u$, and $(\phi_{(U,u)})'(0)>0$.
We again denote $r(U,u)=|(\phi_{(U,u)})'(0)|$ the conformal radius of $U$ with respect to $u$.

By the Riemann Mapping
Theorem, the correspondence $$\iota:(U,u)\mapsto \phi_{(U,u)}$$
establishes a bijection between marked topological disks properly contained in $\CC$ and univalent maps $\phi:\DD\to\CC$
with $\phi'(0)>0$.
The following theorem is due to Carath{\'e}odory, a proof may be found in  \cite{Pom}:

\begin{thm}[{\bf Carath{\'e}odory Kernel Theorem}]
The mapping $\iota$ is a homeomorphism with respect to the Carath{\'e}odory topology on domains and the
compact-open topology on maps.
\end{thm}

\noindent
\begin{prop}
\label{radius-continuous}
The conformal radius of a quadratic Siegel disk varies continuously with respect to the Hausdorff
distance on Julia sets.
\end{prop}

For a pointed domain $(U,u)$ denote $\rho(U,u)$ the {\it inner radius}
$\rho(U,u)=\dist(u,\partial U)$.
\begin{lem}
\label{variation conf radius}
Let $U$ be a simply-connected  bounded subdomain of $\CC$ containing the point $0$ in the interior.
Suppose $V\subset U$ is a simply-connected subdomain of $U$, and $\partial V\subset U_{\eps}(\partial U)$.
Then
$$r(U,0)-r(V,0)\leq 4\sqrt{r(U,0)}\sqrt{\eps}.$$
Moreover, denote $F(x)=4x/(1+x)^2.$ Then
$$r(V,0)\leq r(U,0)F\left(\frac{\rho(V,0)}{\rho(U,0)} \right).$$

\end{lem}

\begin{prop}
\label{r doest drop}
Let $\{\theta_i\}$ be a sequence of Brjuno numbers such that $\theta_i\to\theta$ and
$\overline{\lim}\; r(\theta_i)=l>0$. Then $\theta$ is also a Brjuno number and $r(\theta)\geq l$.
\end{prop}

Let us note for future reference:
\begin{thm}[\cite{BY-book}]
\label{comp-bded}
There exists an algorithm $A$ with an oracle for $\theta$ which, given $\theta$ of a type bounded by $B$ and the value of $B$
uniformly computes $r(\theta)$.
\end{thm}

\begin{thm}[\cite{BBY3}]
\label{comp-bded1}
There exists an algorithm $A$ with an oracle for $\theta$ which, given $\theta$ of a type bounded by $B$, the value of $B$, and $r\in(0,r(\theta))$
uniformly computes the linearizing map $\phi_\theta$ on the disk $B(0,r)$.
\end{thm}

\section{Computation of external angles in the Mandelbrot set}
Suppose that $f_c$ has a fixed point at the origin with multiplier $\lambda=e^{2\pi i \theta}$ for some rational number $\theta=p/q\neq 0$.
Then $c$ lies in the boundary of the main hyperbolic component of the Mandelbrot set $\cM$, and there are exactly two extental
rays of $\cM$ landing at $c$. Let us denote $\alpha$ one of the external angles of $c$.
As described in \cite{milnor-orm}, the angle $\alpha$ is
periodic under the angle doubling map
$$D:x\mapsto 2x \mod 1$$
with period $q$. Furthermore, denote 
$$\cO\equiv (\alpha_0,\alpha_1,\ldots,\alpha_{q-1})$$
the points of the orbit of $\alpha_0=\alpha$ under $D$, enumerated in the cyclic order on $\RR/\ZZ$.
Then 
$$D(\alpha_i)=\alpha_{i+p\mod q},$$
that is, the combinatorial rotation number of the orbit of $\alpha$ on $S^1$ is equal to $p/q$.

As shown in \cite{milnor-orm}, for every combinatorial rotation number $p/q$ there exists a unique $q$-periodic orbit $\cO_{p/q}$
 of $D$ which realizes it.
Let us label the external angles of $c$ by $\alpha_-(p/q)$, $\alpha_+(p/q)$ in the cyclic ordering on $\RR/\ZZ$.
These angles are uniquely determined as elements of $\cO_{p/q}$ such that
\begin{equation}\label{critgap}
\dist(\alpha_-(p/q)-\alpha_+(p/q))=\min\dist(\beta,\gamma)\text{ for }\beta,\gamma\in\cO_{p/q},\;\gamma\neq \beta
\end{equation}
with respect to the Euclidean distance on $\RR/\ZZ$.

\begin{figure}[ht]
\centerline{\includegraphics[width=\textwidth]{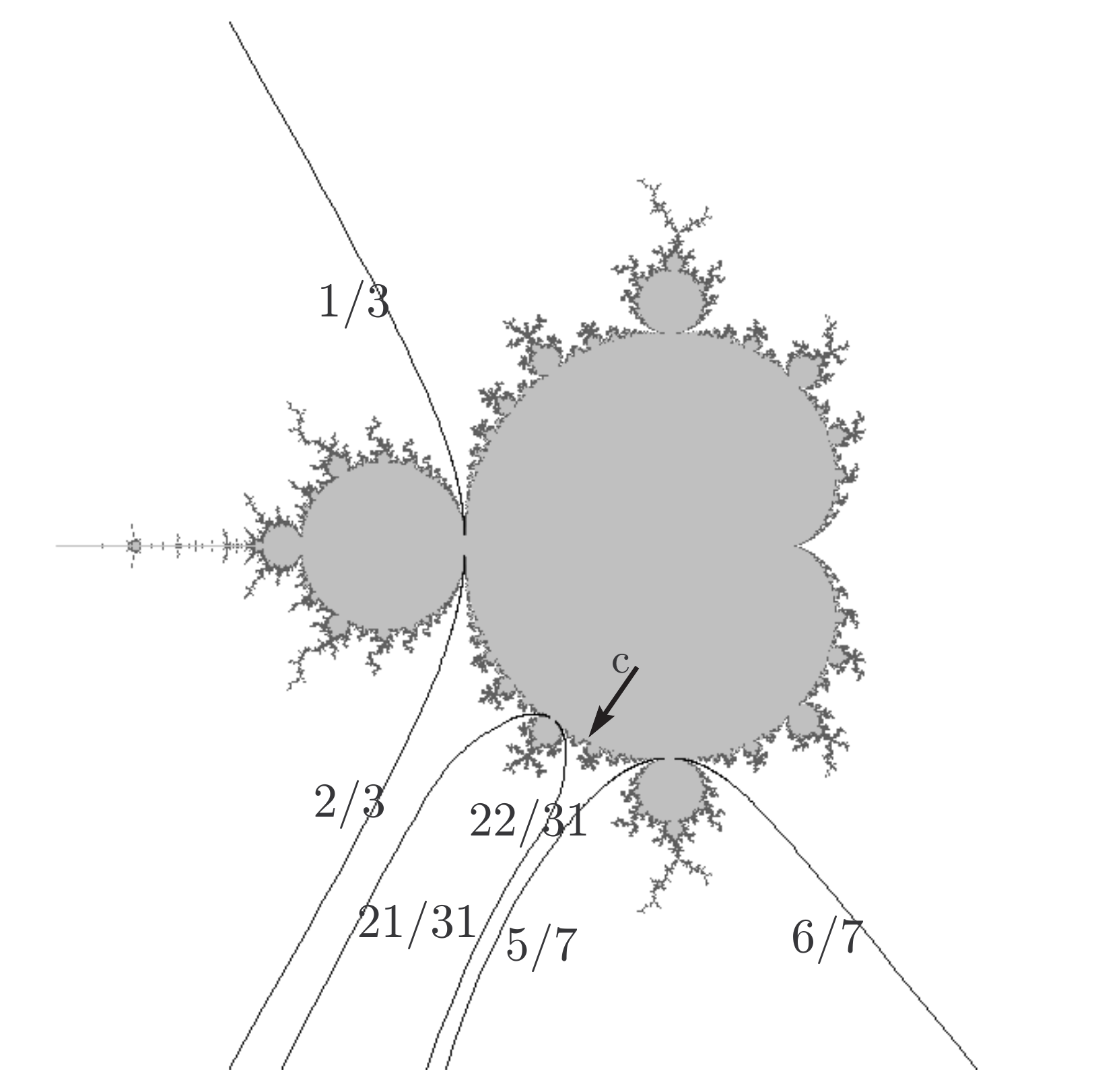}}
\caption{\label{fig-mandel}Approximating the external angle of the parameter $c\in\partial\cM$ for which $f_c$ has a fixed point
with multiplier $\lambda=e^{2\pi i \theta}$ and internal angle $\theta=[1,1,1,\ldots]$ (golden-mean Siegel parameter). The first rational approximants of $\theta$ are:
$\theta_1=1/2$ ($\alpha_-(1/2)=1/3$, $\alpha_+(1/2)=2/3$), $\theta_2=2/3$ ($\alpha_-(2/3)=5/7$, $\alpha_+(2/3)=6/7$) and 
$\theta_3=3/5$ ($\alpha_-(3/5)=21/31$, $\alpha_+(3/5)=22/31$).
}
\end{figure}

Let us now formulate:
\begin{thm}
\label{thm-douady}
Let $f_c$ have a fixed point at the origin with multiplier $\lambda=e^{2\pi i\theta}\in \uc.$ Then the external argument(s)
of $c$ in the Mandelbrot set $M$ is(are) uniformly computable from the continued fraction of $\theta$.
\end{thm}

\begin{proof}
In the case when $\theta=p/q$ is rational, as seen from the above discussion, the angles $\alpha_\pm(\theta)$ belong to the
unique periodic orbit $\cO_{p/q}$ of $D$ with combinatorial rotation number $p/q$. A $q$-periodic point of $D$ in $\RR/\ZZ$ has the form
$k/(2^q-1)$. Thus, we can find the orbit $\cO_{p/q}$  by a finite brute force search, and identify the angles $\alpha_\pm(\theta)$
using (\ref{critgap}). 

Now let $\theta$ be irrational. In this case, the
external angle of $\theta$ is unique, we denote it $\alpha(\theta)$.

 Denote $p_n/q_n$ the continued fraction convergents of $\theta$, and set
$c_n=e^{2\pi ip_n/q_n}$. 
The algorithm $A$ works by computing $\alpha_-(p_n/q_n)$ until
$$|\alpha_-(p_n/q_n)-\alpha_-(p_{n-1}/q_{n-1})|<2^{-n}.$$
It then follows that 
$$|\alpha_-(p_n/q_n)-\alpha(\theta)|<2^{-n}.$$
Indeed, the values $c_n$ and $c_{n-1}$ cut out a boundary arc $L_n$ of the main hyperbolic component of $\cM$ which contains
$c$; and the external angles of all points in $L_n$ are no more than $2^{-n}$ apart.
We illustrate the above approximation process in Figure \ref{fig-mandel}.
\end{proof}

\section{Proof of the Main Theorem}
\subsection{Definition of the Cantor set $C_\theta$}
For $\theta\in \uc \setminus \Q$ let us denote $$c(\theta)=\lambda/2-\lambda^2/4,\text{ where }\lambda=e^{2\pi i\theta},$$
so that $P_\theta(z)=e^{2\pi i\theta}z+z^2$ is conjugate to $f_{c(\theta)}$ by the affine map $z\mapsto z+\lambda/2$.
The parameter $c(\theta)$ lies on the boundary of the main component of the Mandelbrot set. Let us denote $\alpha_\theta$ its
external argument in $M$. As before, denote $D$ the doubling map $x\mapsto 2x\mod 1$.
Let 
$$D^{-1}(\alpha_\theta)=\{\gamma_\theta,\gamma_\theta+\frac{1}{2}\},$$
and denote $L_\theta$ the closed half-circle $\TT=\RR/\ZZ$  with endpoints $\gamma_\theta$, $\gamma_\theta+\frac{1}{2}$, which 
contains $\alpha_\theta$.
\begin{defn}
We denote $C_\theta$ the Cantor set of angles $\alpha$ with the property 
$$D^n(\alpha)\in L_\theta\text{ for }n=0,1,2,\ldots.$$
\end{defn}
The following is known \cite{bullsent94}:
\begin{prop}
\label{properties C}
\begin{enumerate}
\item The dynamics of $D$ is transitive on $C_\theta$.
\item $D$ is injective and preserves circular orientation on $C_\theta$,
 except at the endpoints of $L_\theta$ which are mapped onto a single point $\alpha_\theta$.
 \item The set $C_{\theta}$ consists of the closure of the recurrent point $\alpha_{\theta}$.
\item The map which collapses all gaps of $C_\theta$ semi-conjugates the dynamics of $D$ to that of the rotation of the circle
by angle $\theta$.
\item If $J_{c(\theta)}$ is locally connected, then an external ray $R_\alpha$ lands at a point on the boundary of the Siegel disk
of $f_{c(\theta)}$ if and only if $\alpha\in C_\theta$.

\end{enumerate}
\end{prop}

By \thmref{thm-douady} we have:
\begin{prop}
\label{compute c}
The Cantor set $C_\theta$ is uniformly computable from the continued fraction of $\theta$.
\end{prop}

\begin{proof}We note that by Theorem \ref{thm-douady},  $\alpha_{\theta}$ is uniformly computable from the continued fraction of $\theta$. The iterates $\{D^{n}(\alpha_{\theta})\}_{n\in\NN}$ then form a sequence of uniformly computable points in $C_{\theta}$ which, by Proposition \ref{properties C}, is dense in $C_{\theta}$. On the other hand, the complement of $C_{\theta}$ equals $\bigcup_{n\in\NN}D^{-n}(L_{\theta}^{c})$, where $L_{\theta}^{c}$ denotes the open half-circle with endpoints $\gamma_{\theta}, \gamma_{\theta}+\frac{1}{2}$ which does not contain $\alpha_{\theta}$. Since these endpoints are computable as well as the map $D$, we see that the complement of $C_{\theta}$ is a lower-computable open set. The result now follows from Proposition \ref{computableset}. 
\end{proof}

Denote $\text{Comp}(\CC)$ the set of compact subsets of $\CC$. For $X,Y\in\text{Comp}(\CC)$,
let us denote 
$$d(X,Y)\equiv \inf\{r>0\; |\; X\subset U_r(Y)\}.$$ 
Let $$\Lambda:S\to \text{Comp}(\CC)$$ be a set-valued function on a topological space $S$. We say
that $\Lambda$ is {\it upper semi-continuous} at $s\in S$ if 
$$\forall \{s_n\}\subset S,\; s_n\to s,\text{ we have }d(\Lambda(s_n),\Lambda(s))\to 0.$$
The following is trivially true:
\begin{prop}
\label{upsc}
Let $f_c$ be a quadratic polynomial with a connected Julia set. The dependence of the external ray impression on the angle
$$\be\mapsto \imp(R_\beta)$$
is an upper semi-continuous function of $\uc$.
\end{prop}



Assume now that $f_c$ has a Siegel fixed point $a$ with multiplier $f'(a)=e^{2\pi i\theta}$,
and denote $\Delta$ the Siegel disk of $f_c$. Set $C\equiv C_\theta$. 

\begin{defn}
Let us say that $f_c$ has {\it the small cycles property} if there exists a sequence of periodic orbits
 $$\bar z^i=\{ z_0,z_1=f_c(z_0),z_2=f_c(z_1),\ldots,z_{q-1}=f_c(z_{q-2}) \} $$ such that
\begin{itemize}
\item $d(\bar z^i,\partial\Delta)\to 0$;
\item denote $\alpha_i$ the external angle of $z_i$ and let $p/q$ be the combinatorial rotation number of the cycle $\alpha_0,\ldots,\alpha_{q-1}$,
in other words, let $D(\alpha_i)=\alpha_j$, where $j=i+p\mod q$. Then $p/q\to\theta$. 
\end{itemize}  
\end{defn}
We have the following generalization of the last item of \propref{properties C}:
\begin{prop}
\label{smallcycles}
Suppose $f_c$ has the small cycles property. Then $\imp(R_{\alpha}) \cap \bd\Delta \neq \emptyset$ for every $\alpha \in C$.
\end{prop}
The statement follows from \propref{upsc}.


\subsection{The construction}

We now prove:

\begin{lemma}\label{no comp on boundary}Assume $\Delta_{\theta}$ is a Siegel disk for which
\begin{enumerate}
 \item the linearizing coordinate $\phi:\DD\to\Delta$ continuously extends to is a $C^{1}$-smooth mapping of $S^1\to\bd\Delta$, and
 \item the conformal radius  $r(\theta)$ is not computable. 
\end{enumerate} 
Then every compact set $A\subset J_\theta$ which intersects the boundary $\partial \Delta_{\theta}$ is not computable.
\end{lemma}

\begin{pf} Assume the contrary. Standard facts about continued fractions imply that there exist
$a,b>0$ such that the following holds. Denote $p_n/q_n$ the $n$-th convergent of $\theta$, and let $x\in S^1$.
Let $\tau_\theta(x)=e^{2\pi i\theta}x$ be the rigid rotation.
Then, the points $x,\tau_\theta(x),\ldots,\tau_\theta^{q_n}$ form an $ae^{-bn}$-net of $S^1$.
Let $c>0$ be an upper bound on $|\phi'(x)|$ for $x\in S^1$. Then the set
$$S_n=\cup_{j=1}^{q_n}P_\theta^j(A)$$
has the property
$$d(\partial\Delta,S_n)<cae^{-bn}.$$

It follows that the sets $D_{n}=U_{cae^{-bn}}(S_{n})$ are  connected sets containing $\partial \Delta$.  Let $U_{n}$ be the domain consisting of $D_{n}$ together with the connected component of the complement of the closure of $U_{n}$ which contains the origin.  Clearly, we have that $r(U_{n},0)$ is computable. Moreover, since $\Delta \subset U_{n}$ and $\partial\Delta \subset U_{3 cae^{-bn}}(\partial U_{n})$, by \lemref{variation conf radius} this implies that $r(\theta)$ is also computable, which contradicts our assumptions.
\end{pf}

The statement of the Main Theorem follows immediately from \propref{compute c}, \propref{smallcycles}, \lemref{no comp on boundary}, and the following theorem:
\begin{thm}
\label{pfmain}
For every right computable $r\in(0,r_*)$ there exists a computable Brjuno parameter $\gamma\in\uc$ such that:
\begin{itemize}
\item $r(\gamma)=r$;
\item the boundary $\bd\Delta$ is a $C^1$-smooth curve;
\item $P_\gamma$ satisfies the small cycles condition.
\end{itemize}
Moreover, $\gamma$ is uniformly computable from a computable sequence $r_k\searrow r$.
\end{thm}

The proof of \thmref{pfmain} is a combination of the arguments of \cite{BC2,ABC} and \cite{BY-MMJ}. 
Let $\theta=[r_1,\ldots]\in(0,1)$ be any Brjuno number, and let $p_n/q_n=[r_1,\ldots,r_n]$ denote its continued fraction approximants. 
Let $A>1$, set $A_n=[A^{q_n}]$ and denote 
$$\theta(A,n)=[r_1,\ldots,r_n,A_n,1,1,\ldots].$$
We first state a lemma, which
summarizes Propositions 2 and 3 of \cite{BC2}:
\begin{lem}
\label{stepinduction}
Let $\theta$ and $\theta(A,n)$ be as above. Then,
$$\lim r(\theta(A,n))=\frac{r(\theta)}{A},$$
and the linearizing parametrizations $\phi_{\theta(A,n)}$ converge uniformly to $\phi(\theta)$ on every compact subset of $B(0,r(\theta)/A)$.

Furthermore, $P_{\theta(A,n)}$ has a periodic cycle $O_n$ such that in the Hausdorff topology on compact sets 
$$\lim_{n\to\infty} O_n=\phi_\theta(\partial B(0,r_\theta/A)).$$
The cycle of external angles of rays landing at $O_n$ has a combinatorial rotation number $p_{k_n}/q_{k_n}$. 
\end{lem}

We will also need the following observation (see for instance \cite{milnor-orm}).

\begin{lem}
\label{continuity} Suppose $U$ is an open neighborhood in parameter space $\CC$ such that for all $c\in U$ there exists a periodic repelling  point $x(c)$ which depends continuously on $c$. Then the external angles of rays landing at $x(c)$ do not change through $U$.
\end{lem}

We now prove:
\begin{lem}
\label{ind1}
There exists a sequence of Brjuno numbers
$\gamma_n=[a_1,\ldots,a_{k_n},1,1,1,\ldots]$
such that the following properties hold:
\begin{enumerate}
\item $|\gamma_n-\gamma_{n+1}|<2^{-n}$;
\item $r(\gamma_n)\in (r_n,r_{n+1})$;
\item the $C^1$ distance between the linearizing maps $\phi_{\gamma_n}$ and $\phi_{\gamma_{n+1}}$ is bounded by $2^{-n}$ on the 
closed disk $\overline{B(0,r_{n+2})}$;
\item the map $P_{\gamma_n}$ has a periodic cycle $O_n$ with combinatorial rotation number at infinity equal to $p_{m_n}/q_{m_n}$ with $m_n\leq k_n$ such that
$$\dist_H(O_n,\bd\Delta_{\gamma_n})<2^{-n};$$
\item for every periodic orbit of $P_{\gamma_{n-1}}$ with period $\leq q_{m_{n-1}}$ there is a periodic orbit point of $P_{\gamma_n}$ with the same combinatorial rotation
number at infinity and within Hausdorff distance $2^{-n}$ of it;
\item the number $\gamma_n$ can be computed uniformly from $a_1,\ldots, a_{k_{n-1}}$ and $r_n,r_{n+1}$.
\end{enumerate}
\end{lem}
\begin{proof}
The proof is an induction based on \lemref{stepinduction}. The base of induction is clear. For a step of induction, note that the {\it existence} of 
an angle
$$\gamma_n=[a_1,\ldots,a_{k_{n-1}},\underbrace{1,\ldots,1}_l,N,1,1,\ldots]$$
satisfying the conditions (1)-(5) follows immediately from \lemref{stepinduction} 
and \lemref{continuity}. 

We claim that for each pair $N,l$, conditions   (1)-(5) 
 can be algorithmically checked in the sense that there is an algorithm which, upon input $(N,l)$,
 halts if and only if all the conditions are satisfied by 
$$[a_1,\ldots,a_{k_{n-1}},\underbrace{1,\ldots,1}_l,N,1,1,\ldots].$$ 
Assuming this claim, we can employ an exhaustive search over all pairs $(N,l)$ and and wait until a pair satisfying all the 
conditions is found. Since existence of such a pair is guaranteed, this procedure must eventually halt and we use the found 
pair $N,l$ to define  $\gamma_{n}$, which is then computable from $\gamma_{n-1}$.

We now explain how to algorithmically check 
conditions (1)-(5), proving the claim.  Condition (1) is trivial. By \thmref{comp-bded}, the number $r(\gamma_n)$ is computable 
from $\gamma_{n}$ (a bound for the type of $\gamma_{n}$ can be taken to be for instance $B=\max\{a_{1},\ldots,a_{k_{n-1}},N\}$) and therefore,  
by \thmref{comp-bded1}, so is the the linearizing map $\phi_{\gamma_{n}}$. This allows to check conditions (2) and (3). Finally,
positions of repelling periodic cycles of a given period are easily estimated with an arbitrary precision -- and combinatorial 
rotation numbers of cycles of external rays landing on them are also computable without difficulty. Hence, conditions (4) and (5) are
straightforward to verify algorithmically.
\end{proof}

To finish the proof of \thmref{pfmain}, let $\gamma_n$ be the sequence constructed in \lemref{ind1}, and set 
$$\gamma=\lim \gamma_n.$$

\subsection{Concluding remarks}
We note that several intriguing questions remain open. Firstly, it is natural to ask whether, in the conditions
of Main Theorem, the principal impression $\pimp(R_\alpha)$ is  also non-computable. This seems likely, at least in some 
cases. A more challenging problem is whether there may exist non-computable impressions (with computable external angles)
in the case when the whole Julia set is computable. Our present approach would not be applicable in such a situation.
In the case when a quadratic Julia set is locally connected, it admits an explicit topological model (see the discussion in
\cite{BY-lc}). Coupled with computability of the Julia set, this would rule such quadratics out as a source of examples.
Non locally connected computable Siegel Julia sets as well as Cremer Julia sets (which are always not locally connected
and always computable \cite{BBY1}) may potentially contain non-computable impressions.
However, at present we seem to lack the necessary understanding of the structure of impressions in such sets 
to either present such examples, or to rule them out.

\bibliographystyle{plain}
\bibliography{biblio1}

\end{document}